\documentclass{amsart}

\usepackage{amsmath,amsthm,amsfonts,amssymb,latexsym}

\usepackage{tikz}
\usetikzlibrary{decorations}
\usetikzlibrary{decorations.pathreplacing}
\usetikzlibrary{decorations.markings}

\newcommand{\fp}{\mathop{\scalebox{1.2}{\raisebox{0ex}{$\ast$}}}}
\newcommand{\fpGammaiI}{\fp_i \Gamma_i}

\newtheorem{thm}{Theorem}[section]
\newtheorem{thm*}{Theorem}
\newtheorem{cor}[thm]{Corollary}
\newtheorem{lemma}[thm]{Lemma}
\newtheorem{remark}[thm]{Remark}
\newtheorem{lemma*}[thm*]{Lemma}
\newtheorem{prop}[thm]{Proposition}

\newtheorem{question}[thm]{Question}

\newtheorem{defn}[thm]{Definition}

\DeclareMathOperator{\dom}{dom}
\DeclareMathOperator{\ran}{ran}

\newcommand{\from}{\colon}

\newcommand{\baire}{\N^{\N}}
\newcommand{\N}{\mathbb{N}}
\newcommand{\R}{\mathbb{R}}
\newcommand{\F}{\mathbb{F}}

\newcommand{\Z}{\mathbb{Z}}

\renewcommand{\subset}{\subseteq}

\newcommand{\union}{\cup}
\newcommand{\bigunion}{\bigcup}
\newcommand{\inters}{\cap}

\newcommand{\<}{\langle}
\renewcommand{\>}{\rangle}

\DeclareMathOperator{\AD}{AD}
\DeclareMathOperator{\AC}{AC}

\newcommand{\define}[1]{\textbf{#1}} 

\usepackage{color}

\title{Hyperfiniteness and Borel combinatorics}

\author[Conley]{Clinton T. Conley}
\author[Jackson]{Steve Jackson}
\author[Marks]{Andrew S. Marks}
\author[Seward]{Brandon Seward}
\author[Tucker-Drob]{Robin D. Tucker-Drob}

\thanks{This paper was produced as the result of a SQuaRE workshop at the
American Institute of Mathematics. The first author is supported by NSF
grant DMS-1500906. The
third author is supported by NSF grant DMS-1500974. The fourth author is
partially supported by ERC grant 306494 and by Simons Foundation grant
328027 (P.I. Tim Austin). The fifth author is supported by NSF grant
DMS-1600904.}

\DeclareMathOperator{\Free}{Free}

\date{\today}

\begin{document}

\begin{abstract}
We study the relationship between hyperfiniteness and problems in Borel
graph combinatorics by adapting game-theoretic techniques introduced by
Marks to the hyperfinite setting. We compute the possible Borel chromatic
numbers and edge chromatic numbers of bounded degree acyclic hyperfinite
Borel graphs and use this to answer a question of Kechris and Marks about
the relationship between Borel chromatic number and measure chromatic
number. We also show that for every $d > 1$ there is a $d$-regular acyclic
hyperfinite Borel bipartite graph with no Borel perfect matching. These
techniques also give examples of hyperfinite bounded degree Borel graphs
for which the Borel local lemma fails, in contrast to the recent results of
Cs\'oka, Grabowski, M\'ath\'e, Pikhurko, and Tyros. 

Related to the Borel Ruziewicz problem, we show there is a continuous
paradoxical action of $(\Z/2\Z)^{*3}$ on a Polish space that admits a
finitely additive invariant Borel probability measure, but admits no
countably additive invariant Borel probability measure. In the context of studying
ultrafilters on the quotient space of equivalence relations under $\AD$, 
we also construct an ultrafilter $U$ on the
quotient of $E_0$ which has surprising complexity. In particular, Martin's measure is Rudin-Kiesler reducible to
$U$. 

We end with a problem about whether every hyperfinite bounded degree Borel
graph has a witness to its hyperfiniteness which is uniformly bounded below
in size. 
\end{abstract}

\maketitle

\section{Introduction}

In this paper, we investigate the relationship between hyperfiniteness and
problems in Borel graph combinatorics. Recall that a (simple) \define{Borel
graph} $G$ on a standard Borel space $X$ is a graph whose vertex set is $X$
and whose (symmetric irreflexive) edge relation is Borel. $G$ is said to be
\define{hyperfinite} if it can be written as an increasing union of Borel
graphs with finite connected components. Hyperfinite graphs can be thought
of as the simplest graphs that can display nonclassical behavior in the setting
of Borel graph combinatorics. This is made precise by the
Glimm-Effros dichotomy.

A fundamental theorem of Kechris, Solecki, and Todorcevic {\cite[Proposition
4.6]{KST}} states that every Borel graph $G$ of degree at most $d$ has Borel
chromatic number $\chi_B(G) \leq d +1$, where the \define{Borel chromatic
number} $\chi_B(G)$ of $G$ is the least cardinality of a Polish space $Y$
so that there is a Borel $Y$-coloring of $G$. This bound is optimal even
for acyclic graphs since for every $d \geq 1$ and $k \in \{2,\ldots,
d+1\}$, there is an acyclic $d$-regular Borel graph with $\chi_B(G) = k$ by
\cite{M1}. However, the graphs used to obtain this result are not
hyperfinite, and Conley and Miller have asked whether every acyclic bounded
degree hyperfinite Borel graph $G$ has $\chi_B(G) \leq 3$ {\cite[Problem
5.17]{KMa}}. We answer this question in the negative. Essentially, we
reprove all of the combinatorial results from \cite{M1} about Borel
colorings, edge colorings, matchings, etc. for Borel graphs with the
additional property of hyperfiniteness. Hence, among bounded degree Borel
graphs, even hyperfinite graphs can achieve the maximum possible
combinatorial complexity as measured by how hard they are to color and
match in a Borel way. This is in contrast to the measure-theoretic context,
where hyperfinite bounded degree Borel graphs are known to be much simpler
to measurably color than arbitrary bounded degree Borel graphs.
For instance, every acyclic hyperfinite bounded degree graph on a standard
probability space $(X,\mu)$ has a $\mu$-measurable $3$-coloring
{\cite[Theorem A]{CM}}. 

To prove these results, we associate to each countable discrete group
$\Gamma$ a certain hyperfinite Borel action of $\Gamma$. 
We then show that an analogue of the central lemma of \cite{M1} is true for
these actions. Recall that if a group $\Gamma$ acts on a set $X$, the free
part of this action is $\Free(X) = \{x \in X : \forall \gamma \in \Gamma(
\gamma \neq 1 \implies \gamma \cdot x \neq x)\}$. 

\begin{defn}\label{action_defn}
Suppose $\Gamma$ is a countable discrete group. Then $\Gamma$ acts on 
$\Gamma^\Gamma$ by 
\[(\gamma \cdot x)(\delta) = \gamma x(\gamma^{-1}\delta)\]
for every $x \in \Gamma^\Gamma$ and $\gamma,\delta \in \Gamma$. 
Let 
$H(\Gamma^\Gamma)$ be the set of $x \in \Free(\Gamma^\Gamma)$ such that $x$ is a bijection and the
permutation $x$ induces on $\Gamma$ has one orbit. 
Let $E_\Gamma$ be the orbit equivalence relation of this action of $\Gamma$
on $H(\Gamma^\Gamma)$. Let $w
\from H(\Gamma^\Gamma) \to H(\Gamma^\Gamma)$ be the Borel function defined by 
$w(x) = (x(1))^{-1} \cdot x$. 
\end{defn}

Note that this action, which we use throughout the paper, is \emph{not} the
standard shift action. It is a combination of the shift action, and 
pointwise multiplication.

An easy calculation shows that $w^{n}(x) = (x^{n}(1))^{-1} \cdot x$ for every $n
\in \Z$.
Thus, since the permutation $x$ induces on $\Gamma$ has a single
orbit, $w$ generates $E_\Gamma$, which is therefore hyperfinite 
{\cite[Theorem 6.6]{KMi}}:

\begin{prop}
  $E_\Gamma$ is hyperfinite. \qed
\end{prop}

We prove the following version of {\cite[Lemma 2.1]{M1}} for these
hyperfinite actions:

\begin{lemma}\label{hyp_lemma}
  Suppose $\Gamma$ and $\Delta$ are countable groups and $A \subset
  H((\Gamma * \Delta)^{\Gamma * \Delta})$. Then 
  \begin{enumerate}
  \item There is an injective Borel 
  $\Gamma$-equivariant function $f \from H({\Gamma}^{\Gamma}) \to
  H((\Gamma * \Delta)^{\Gamma * \Delta})$ with $\ran(f) \subset A$, or
  \item There is an injective Borel 
  $\Delta$-equivariant function $f \from H({\Delta}^{\Delta}) \to
  H((\Gamma * \Delta)^{\Gamma * \Delta})$ with $\ran(f) \inters A =
  \emptyset$. 
  \end{enumerate}
\end{lemma}

By applying this lemma the same way as \cite{M1} we obtain hyperfinite
versions of all the theorems in that paper, as illustrated in
Theorem~\ref{comb_thm}. Recall that a graph is said to
be \define{$d$-regular} if all of its vertices have degree $d$. 

\begin{thm}\mbox{ }\label{comb_thm}
\begin{enumerate}
  \item For every $d \geq 1$ and every $k \in \{2, \ldots, d+1\}$ there is
  a $d$-regular acyclic hyperfinite Borel graph $G$ with $\chi_B(G) = k$.

  \item For every $d \geq 1$ and every $k \in \{d, \ldots, 2d - 1\}$ there
  is a $d$-regular acyclic Borel bipartite hyperfinite graph $G$ such that
  $\chi'_B(G) = k$.

  \item For every $d > 1$ there exists a $d$-regular acyclic hyperfinite
  Borel bipartite graph with no Borel perfect matching.

\end{enumerate}
\end{thm}

Part (1) of this theorem answers a question of Conley and Miller
{\cite[Question 5.17]{KMa}}.

By combining part (1) of Theorem~\ref{comb_thm} with the result from
\cite{CM} that every bounded degree acyclic hyperfinite Borel graph $G$ has $\chi_M(G)
\leq 3$ we also obtain the following, answering a question of Kechris and
Marks {\cite[Question 6.4]{KMa}} (see \cite{KMa} for a definition of the measure chromatic
number $\chi_M$).

\begin{cor}
  For every $d \geq 1$ and every $k \in \{2, \ldots, d+1\}$, there is
  a $d$-regular acyclic Borel graph $G$ with $\chi_B(G) = k$ and $\chi_M(G) = 3$.
\end{cor}

Cs\'oka, Grabowski, M\'ath\'e, Pikhurko, and Tyros
have recently proved a Borel version of the local lemma for
bounded degree Borel graphs of uniformly subexponential growth
\cite{CGMPT}. We give the
precise statement of their theorem in Section~\ref{lll_subsection}. One might
hope that the Borel version of the local lemma is true for all
hyperfinite bounded degree Borel graphs. (Note that every Borel graph of
uniformly polynomial growth is hyperfinite by \cite{JKL}, and it is open
whether every bounded degree Borel graph of uniformly subexponential growth
is hyperfinite). We show that the Borel local lemma may fail for
hyperfinite Borel graphs:

\begin{thm}\label{lll_thm}
  There is a hyperfinite bounded degree  Borel graph $G$ so that the
  Borel local lemma in the sense of \cite{CGMPT} is false for $G$. 
\end{thm}

The examples we give are graphs generated by free hyperfinite actions of
$\F_{2n}$ for $n \geq 6$. The proof uses an idea of Kechris and Marks for
constructing Borel graphs for which the local lemma fails using the results
of \cite{M1}. 

Recall that if $X$ is a Polish space and $\mathcal{B}(X)$ are the Borel
subsets of $X$, then a \define{finitely additive Borel probability measure}
on $X$ is a finitely additive function $\mu \from \mathcal{B}(X) \to [0,1]$
such that $\mu(X) = 1$.
Lemma~\ref{hyp_lemma} can also be used to show the existence of certain
exotic finitely additive invariant Borel probability measures. 
This is interesting in light of
the Borel Ruziewicz problem:
whether Lebesgue measure is the only finitely additive
isometry-invariant probability measure defined on the Borel subsets of the $n$-sphere
for $n \geq 2$ {\cite[Question 11.13]{W}}. By results of 
Margulis and Sullivan ($n \geq 4$) and Drinfeld ($n = 2, 3$)
\cite{D}\cite{M}\cite{S}
it is known that any such measure not equal to Lebesgue measure must fail to be absolutely continuous
with respect to Lebesgue measure. Furthermore, by a result of Dougherty and Foreman, it
is also known that any such measure must be supported on a meager subset
of $X$
\cite{DF}. 
Generalizing this last result, Marks and Unger have shown that if any
group $\Gamma$ acts by Borel automorphisms paradoxically on a Polish space
$X$, then any
finitely additive $\Gamma$-invariant Borel measure on $X$ must be supported
on a meager subset of $X$ \cite{MU}. 

It has been an open problem to find any paradoxical Borel action of
a group on a standard Borel space
that admits an ``exotic'' finite additive invariant Borel probability
measure (in particular, one that is not countably additive). We show the following, where the group $(\Z/2\Z)^{*3}$ is a free
product of three copies of $\Z/2\Z$:

\begin{thm}[$\AC$]\label{exotic_fin_measure}
  There is a continuous free action of
  $(\Z/2\Z)^{*3}$ (which is hence paradoxical) on a Polish space 
  so that this action admits a finitely additive invariant
  Borel probability measure, but does not admit any countably additive
  invariant Borel probability measure. 
\end{thm}

Our techniques also allow us to construct interesting measures in a
different context. Zapletal has suggested the problem of investigating the
structure of ultrafilters on $2^\N/E_0$ under $\AD$. Some examples of such
ultrafilters are the ultrafilter $U_L$ containing the Lebesgue conull
$E_0$-invariant sets, and the ultrafilter $U_C$ containing the comeager
$E_0$-invariant sets.
One can organize such ultrafilters by Rudin-Kiesler reducibility. Here, for
example, it is open whether every ultrafilter on $2^\N/E_0$ is
Rudin-Kiesler above $U_L$ or $U_C$. (See {\cite[Section 4]{M2}}
for further discussion and a definition of Rudin-Kiesler reducibility). We show the existence of an 
ultrafilter on $2^\N/E_0$ which has surprising complexity:

\begin{thm}[$\AD$] \label{uf} There is an ultrafilter $U$ on $2^\N/E_0$ so
that Martin measure on $2^\N/\equiv_T$ is Rudin-Kiesler reducible to $U$.
In fact, the Rudin-Kiesler reduction can be chosen to be Borel. 
\end{thm}

It is an open question whether there is a nontrivial ultrafilter on
$2^\N/E_0$ that is Rudin-Kiesler reducible to Martin's ultrafilter. This is
equivalent to Thomas's question of whether 
Martin measure is strongly ergodic \cite{T}.

One way of regarding Theorem~\ref{comb_thm}
is that it constrains 
types of witnesses that can exist to the hyperfiniteness of a Borel graph. In
particular, there do not exist witnesses to hyperfiniteness possessing
properties useful for carrying out combinatorial constructions. In
contrast, in the measure-theoretic or Baire category contexts, there do
exist ``nice'' hyperfiniteness witnesses. 
See for example \cite{CM} for such constructions. 

We pose an open question that is a very simple
attempt to understand what global control we can exert over the witnesses
to the hyperfiniteness of a bounded degree Borel graph:

\begin{question}\label{q1}
  Suppose $G$ is a bounded degree hyperfinite Borel graph. Does there exist
  an increasing sequence $G_1 \subset G_2 \subset \ldots \subset G$ of Borel 
  subgraphs of $G$ such that
  \begin{enumerate}
    \item $G_1, G_2, \ldots$ witnesses that $G$ is hyperfinite. That is,
    for every $n$, each
    connected component of $G_n$ is finite, and $\bigunion_{n} G_n = G$.
    \item Every connected
    component of $G_n$ has cardinality at least $n$. 
  \end{enumerate}
\end{question}

We show that this question has a positive answer in some contexts:
\begin{prop}\label{q1_prop}
Suppose $G$ is a bounded degree hyperfinite Borel graph on a standard Borel
space $X$. Then
Question~\ref{q1} has a positive answer modulo a nullset with respect to
any Borel probability measure on $X$ and a meager set with respect to any
compatible Polish topology on $X$. Question~\ref{q1} also has a positive
answer when $G$ is generated by a single Borel function.
\end{prop}

However, we conjecture that Question~\ref{q1} has a negative answer. 

\subsection{Notation and conventions}

Our notation is mostly standard, and largely follows \cite{KMa}. Ideally,
the reader will also have some familiarity with \cite{M1}, since
much of what follows builds on ideas from that paper. 

\section{The main lemma for $H((\Gamma * \Delta)^{\Gamma * \Delta})$}

Suppose $\Gamma$ is a countable discrete group.
Throughout this section we will often deal with partial functions from
$\Gamma$ to $\Gamma$. We may define the same action as in
Definition~\ref{action_defn} more generally for partial functions, and we
begin by defining an associated partial order: 

\begin{defn}
  Suppose $x$ is a partial injection from $\Gamma$ to $\Gamma$. If
  $1 \in \dom(x)$, then define $w(x) = (x(1))^{-1} \cdot x$, otherwise
  $w(x)$ is undefined. Define a strict partial order $<_\Gamma$ on the space of
  partial functions from $\Gamma$ to $\Gamma$ by setting $x <_\Gamma y$ if
  $\exists n > 0 (w^n(x) = y)$. Since $w$ generates $E_\Gamma$, on each
  $E_\Gamma$-class, the ordering $<_\Gamma$ is isomorphic to a subordering
  of $\Z$. 
\end{defn}

Next, we make some additional definitions related to
$H(\Gamma^\Gamma)$: 

\begin{defn}
  Suppose $x$ is a finite partial injection from $\Gamma$ to $\Gamma$.
  Say $x$ has
  \define{one
  orbit} if for all $\gamma, \delta \in \dom(x)$ there is an $n \in \Z$ such
  that $x^n(\gamma) = \delta$. If $x$ is nonempty, say \define{$x$ begins 
  at $\gamma$} if $\gamma \in \dom(x)$ but $\gamma \notin ran(x)$
  and \define{$x$ ends at $\delta$} if $\delta \in \ran(x)$ but $\delta
  \notin \dom(x)$. If $x$ is the empty function, then say that $x$ begins
  and ends at $1$.
\end{defn}

Note that the action of $\Gamma$ on $\Gamma^\Gamma$ in
Definition~\ref{action_defn} is chosen to interact well with the
permutation each bijection $x \in \Gamma^\Gamma$ induces on $\Gamma$. In
particular, suppose $y$ is a partial function from $\Gamma$ to $\Gamma$ and
$R \subset \Gamma$ is an orbit of $y$. Then it is easy to check that for
every $\gamma \in \Gamma$, $\gamma R$ is an orbit of the
permutation induced by $\gamma \cdot y$. 

We are now ready to prove our main lemma.

\begin{proof}[Proof of Lemma~\ref{hyp_lemma}.]
  We may assume that $\Gamma$ and $\Delta$ are nontrivial. As in
  the proof of {\cite[Lemma 2.1]{M1}}, let $Y \subset (\Gamma *
  \Delta)^{\Gamma * \Delta}$ be the set of all $x \in (\Gamma *
  \Delta)^{\Gamma * \Delta}$ such that for all $\alpha \in \Gamma * \Delta$
  and all nonidentity $\gamma \in \Gamma$ and $\delta \in \Delta$ we have
  $\gamma \cdot (\alpha^{-1} \cdot x) \neq (\alpha^{-1} \cdot x)$ and
  $\delta \cdot (\alpha^{-1} \cdot x) \neq (\alpha^{-1} \cdot x)$. Note
  that $\Free((\Gamma * \Delta)^{\Gamma * \Delta}) \subset Y$. 

  Every nonidentity word $\alpha \in \Gamma$ can be written as a reduced
  word of the form $\gamma_0 \delta_0 \gamma_1 \delta_1 \ldots$ or
  $\delta_0 \gamma_0 \delta_1 \gamma_1 \ldots$ where $\gamma_i \in \Gamma$
  and $\delta_i \in \Delta$ are nonidentity elements. We let the
  \define{length} of $\gamma \in \Gamma * \Delta$ be its length as a
  reduced word. We say $\alpha \in \Gamma * \Delta$ is a
  \define{$\Gamma$-word} if it begins with an element of $\Gamma$ as a
  reduced word, and a \define{$\Delta$-word} if it begins with an element
  of $\Delta$ as a reduced word. So $\Gamma * \Delta$ is the disjoint union
  of the set of $\Gamma$-words, $\Delta$-words, and the identity.

For each $B \subset Y$, 
define a game $G_B$ for producing a (perhaps partial) injection $y$ from
  $\Gamma * \Delta$ to $\Gamma * \Delta$ with one orbit. The players will alternate
  defining $y(\alpha)$ for finitely many $\alpha \in \Gamma * \Delta$ subject to the following
  rules:
  
  \begin{itemize}
  \item 
  After each move of player I, $y$ must be injective, have one orbit, and 
  end at some $\Gamma$-word. After each move of player II, $y$ must be
  injective, have one orbit, and end at some $\Delta$-word. 
  \item On each move of the game, if the current partial function $y$
  that has been defined before this move ends at $\xi \in \Gamma *
  \Delta$, then as part of the current move, the current player must define
  $y(\xi)$.
  \item In addition to the requirement of the previous rule, on each of their moves player I may also 
  define $y(\alpha)$ for arbitrarily many $\alpha$ that are
  $\Gamma$-words. On each of their moves player II may also define $y(\alpha)$ for arbitrarily
  many nonidentity $\alpha$ that are $\Delta$-words. 

  \item 
  At the end of the game, if $y$ is not a total function, then II loses if
  and only if among the $\alpha \notin \dom(y)$ that are of minimal length,
  there is some $\alpha$ which is a $\Delta$-word or the identity. If $y$ is total but $y \notin Y$, then II loses if and only if among the
  $\alpha$ witnessing $y \notin Y$ of minimal length, there is some
  $\alpha$ which is a $\Delta$-word, or $\alpha = 1$ witnesses $\alpha
  \notin Y$ via the fact that $\delta \cdot y = y$ for some nonidentity
  $\delta \in \Delta$. Finally, if $y$ is total and $y \in Y$, then I wins
  if $y$ is not in $B$. 
  \end{itemize}

  For example, on the first turn of the game (where our current version
  of $y$ is the empty function which by definition ends at $1$),
  player I must define $y(1)$, and then may also define $y$ on finitely
  many other $\Gamma$-words. The resulting finite partial function $y$ must
  be injective, have one orbit, and end at some $\Gamma$-word. 

  Note that since
  $y$ has a single orbit after each turn of the game, it will also have a
  single orbit at the end of the game.

  By {\cite[Lemma 2.3]{M1}} we can find a Borel subset $C$ of $Y \setminus
  \Free((\Gamma * \Delta)^{\Gamma * \Delta})$ such that $C$ meets every
  $E_\Delta$-class on $Y \setminus
  \Free((\Gamma * \Delta)^{\Gamma * \Delta})$ and the complement of $C$ meets
  every $E_\Gamma$-class on $Y \setminus
  \Free((\Gamma * \Delta)^{\Gamma * \Delta})$.

  By Borel determinacy, one of the two players must have a winning strategy
  in the game associated to the set $B = A \union C$.
  Suppose player I has a winning strategy, and fix such a
  strategy. We will construct
  an injective Borel $\Delta$-equivariant function $f \from
  H(\Delta^\Delta) \to H((\Gamma * \Delta)^{\Gamma * \Delta})$ with
  $\ran(f) \inters A = \emptyset$. 
  We will define $f$ so that for all $x \in H(\Delta^\Delta)$, $f(x)$ is a
  winning outcome of player I's winning strategy in the game and so $f(x)
  \notin A$. 
  We
  will ensure that $f$ is injective by enforcing that $x <_\Delta y$ if and
  only if $f(x) <_{\Gamma * \Delta} f(y)$.
  

  Let $E_0 \subset E_1 \subset E_2 \subset \ldots$ be finite Borel
  equivalence relations that witness the hyperfiniteness of
  $E_\Delta$. We may assume that $E_0$ is the equality relation and
  also that every $E_n$-class is an interval in the ordering $<_\Delta$ by passing
  instead to the relations $E'_n$ where $x \mathrel{E'_n} y$ if the
  $<_\Delta$-interval from $x$ to $y$ lies inside $[x]_{E_n}$.
  For each $x \in H(\Delta^\Delta)$, let $E^x_n$ be the equivalence relation on
  $\Gamma * \Delta$ where $\alpha_0 \mathrel{E^x_n} \alpha_1$ if and only if
  $\delta_0^{-1} \cdot x \mathrel{E_n} \delta_1^{-1} \cdot x$ where $\delta_0,
  \delta_1$ are the unique elements of $\Delta$ such that $\alpha_0$ and
  $\alpha_1$ can be expressed as $\alpha_0 =
  \delta_0 \beta_0$ and $\alpha_1 = \delta_1 \beta_1$ where $\beta_0$ and
  $\beta_1$ are $\Gamma$-words or the identity. Note that in general the
  classes of $E^x_n$ will not be finite. However, for each $E^x_n$-class
  $[\alpha]_{E^x_n}$, $[\alpha]_{E^x_n} \inters \Delta$ will be finite. 

  Fix $x \in H(\Delta^\Delta)$. We will define $f(x)$ via a construction
  that takes countably many steps. In this construction, for each $\delta \in \Delta$ we will play an instance of the game whose outcome
  will be equal to $\delta^{-1}
  \cdot f(x)$. At step $0$ of our construction, let player I move in the game
  associated to each $\delta \in \Delta$ using their winning strategy. 
  
  The only choices we will be making in our construction (other than
  keeping the play of the games consistent with each other) will be
  connecting up the orbits of the finite partial functions constructed in each of the
  games
  so that everything is eventually connected. We will do this
  using the witness to the hyperfiniteness of $E_\Delta$ and in our
  role as player II in all the games. 

  Inductively assume that after step $n$ of our construction, for every
  $\delta \in \Delta$,
  \begin{enumerate}
    \item $f(x) \restriction [\delta]_{E^x_n}$ is a
    finite partial injection which has one orbit.
    \item If
    $\delta_0, \ldots, \delta_k$ enumerates the elements of
    $[\delta]_{E^x_n} \inters \Delta$ in the order so that $\delta_0^{-1}
    \cdot x
    <_\Delta \ldots <_\Delta \delta_k^{-1} \cdot x$, then 
    $f(x) \restriction [\delta]_{E^x_n}$ ends at a
    group element of the form $\delta_k \beta$, where 
    $\beta$ is a $\Gamma$-word. 
    \item For every $\delta_i, \delta_j \in [\delta]_{E^x_n} \inters \Delta$
    such that $\delta_i^{-1}
    \cdot x \mathrel{E_n} \delta_j^{-1} \cdot x$, we have that $\delta_i^{-1}
    \cdot x <_\Delta \delta_j^{-1} \cdot x$ if and only if $\delta_i^{-1}
    \cdot f(x) <_{\Gamma * \Delta} \delta_j^{-1} \cdot f(x)$.
    \item In the game associated to $\delta$, the last move was made by
    player I (using their strategy). The current finite
    partial function defined in the game associated to $\delta$
    includes every value of $(\delta^{-1} \cdot
    f(x))(\alpha)$ we defined during the previous step of the construction, where $\alpha$ is a $\Gamma$-word or the identity.
    \item 
    If $n > 0$ and $[\delta^{-1} \cdot x]_{E_n}$ contains more than one element, then
    in the game associated to $\delta$, during step $n$ we played a move
    for both player II and player I (in that order), and we defined
    $(\delta^{-1} \cdot f(x))(\alpha)$ for every $\Delta$-word $\alpha$
    contained in $\delta^{-1} [\delta]_{E^x_{n-1}}$ that had been already defined in
    step $n-1$.
  \end{enumerate}
  We now describe step $n+1$ of our construction. For each $\delta \in
  \Delta$, let $\delta_0, \ldots, \delta_k$ enumerate the elements of
  $[\delta]_{E^x_{n+1}} \inters \Delta$ in the order so that $\delta_0^{-1}
  \cdot x <_\Delta \ldots <_\Delta \delta_k^{-1} \cdot x$.
  Each $E^x_{n+1}$-class
  $[\delta]_{E^x_{n+1}}$ contains finitely many $E^x_{n}$-classes
  $[\beta_0]_{E^x_{n}}, \ldots, [\beta_m]_{E^x_{n}}$. For every $i < k$
  so that $\delta_i$ and $\delta_{i+1}$ are in different
  $E^x_{n}$-classes, 
  $f(x) \restriction [\delta_{i}]_{E^x_{n}}$ ends at some
  group element $\xi_{i}$ and $f(x) \restriction [\delta_{i+1}]_{E^x_{n}}$
  begins at some group element $\alpha_{i+1}$. In this case, define $f(x)(\xi_i) =
  \alpha_{i+1}$ so that 
  \[w(\xi_{i}^{-1} \cdot f(x)) = 
  \alpha_{i+1}^{-1} \cdot f(x).\]
  After doing this, note that parts (1), (2), and (3) of our induction
  hypothesis are true for $n+1$. However, we are not yet
  finished with our definition of $f(x)$ at step $n+1$ so these properties
  still need to be checked after we are finished. 

  Assume now that $k \geq 1$ so there are at least two elements of
  $[\delta]_{E^x_{n+1}} \inters \Delta$ (else we are finished with our
  definition of $f(x) \restriction [\delta]_{E^x_{n+1}}$ and part (5) of
  our induction hypothesis is also true).
  For each $\delta_i$, one at a time and in order, we
  will first move for player II in the game associated to $\delta_i$, and then let the
  strategy for player I move. Before we begin this process, note that by
  the previous paragraph, $f(x) \restriction
  [\delta]_{E^x_{n+1}}$ has one orbit and
  ends at a word of the form $\delta_k \beta$ where $\beta$ is a $\Gamma$-word. Indeed, inductively, before we consider the game associated to
  $\delta_i$, it will be the case that $f(x) \restriction
  [\delta]_{E^x_{n+1}}$ has one orbit and
  ends at a word of the form $\delta_{i'} \beta$ where $i' = i-1
  \bmod{k+1}$ and $\beta$ is a $\Gamma$-word. 
  
  So for the game
  associated to each $\delta_i$, we make a move for player II by playing
  every value of 
  $\delta_i^{-1} \cdot (f(x) \restriction
  [\delta]_{E^x_{n+1}})$ that has already been defined but not yet played in
  the game. Playing these values will be consistent with the 
  rules of the game by our induction hypothesis. Now we let player
  I's strategy move in the game to define additional values of $\delta_i^{-1} \cdot
  f(x)$. Note that after these two moves, $f(x) \restriction [\delta]_{E^x_{n+1}}$ will have one
  orbit and end at a group
  element of the form $\delta_i \beta$ where $\beta$ is a $\Gamma$-word by
  the rules of the game. 
  After doing this for each of $\delta_0, \ldots, \delta_k \in 
  [\delta]_{E^x_{n+1}} \inters \Delta$ in order, we are finished with 
  step $n+1$ of the construction. Verifying that our inductive hypotheses
  are satisfied is easy, and we are now done with the
  construction of $f(x)$. Verifying that $f$ is Borel and
  $\Delta$-equivariant is straightforward.

  Suppose $x \in H(\Delta^\Delta)$. 
  By part (5) of our induction hypothesis, since $[x]_{E_n}$ has at least
  two elements for sufficiently large $n$, we will play infinitely many
  moves in the game associated to $\delta = 1$ (so the game finishes), and the outcome of the game
  will be equal to the value of $f(x)$ defined by our construction by (4)
  and (5).
  The only thing that remains is to verify that $f(x)$ is total and 
  $f(x) \in H((\Gamma * \Delta)^{\Gamma * \Delta})$ for
  every $x \in H(\Delta^\Delta)$. 
  
  To begin, we prove that $f(x)$ is total for every $x \in
  H(\Delta^\Delta)$. Note that by the definition of
  step 0 of our construction $f(x)(1)$ is defined for all $x \in
  H(\Delta^\Delta)$ (since on the first move player I must define $y(1)$). 
  Then inductively, supposing $f(x)(\alpha)$ is defined on all words
  $\alpha$ of length $n$,
  if $\alpha = \delta \beta$ is any
  $\Delta$-word of length $n+1$ where $\delta \in \Delta$ and $\beta$ is a
  $\Gamma$-word or the identity, then $f(x)(\alpha) =
  \delta ((\delta^{-1} \cdot f(x))(\beta)) = \delta (f(\delta^{-1} \cdot
  x)(\beta))$ must be
  defined by our induction hypothesis. Thus, $f(x)(\alpha)$ must be defined
  for all $\Delta$-words of length $n+1$, and thus also defined 
  for all $\Gamma$-words of length $n+1$ (else player I loses).

  Similarly, the same inductive idea shows that $f(x) \in Y$ for every $x \in
  H(\Delta^\Delta)$. (Alternatively, copy the penultimate paragraph of the proof of
  {\cite[Lemma 2.1]{M1}}).
  
  Now since we have proved that $f(x) \in Y$ for every $x \in
  H(\Delta^\Delta)$, we must have that $f(x) \in \Free((\Gamma *
  \Delta)^{\Gamma * \Delta})$ since $C$ meets every
  $\Delta$-invariant set in $Y \setminus \Free((\Gamma * \Delta)^{\Gamma *
  \Delta})$ by the definition of $C$. Thus, $f(x) \in H(
  (\Gamma * \Delta)^{\Gamma * \Delta})$ for every $x \in H(\Delta^\Delta)$.
  Finally, $f(x) \notin A$ since $f(x)$ is a winning outcome of player I's
  strategy in $G_{A \union C}$. 
  
  This completes the proof
  in the case that player I has a winning strategy in $G_{A \union C}$. The
  proof in the case that player II has a winning strategy is very
  similar.
\end{proof}

Various bells and whistles can be added onto the above lemma. 
For example, the generalization of this lemma to countable
free products is also true:

\begin{lemma}\label{ctbl_version}
  Suppose $I \in \{2, 3, \ldots, \N\}$, $\{\Gamma_i\}_{i \in I}$ is a collection of
  countably many countable discrete groups, and $\{A_i\}_{i \in I}$ is a Borel partition
  of $H((\fpGammaiI)^{\fpGammaiI})$. Then there exists some $j \in I$ and an injective Borel
  $\Gamma_j$-equivariant function $f \from H({\Gamma_j}^{\Gamma_j}) \to
  H((\fpGammaiI)^{\fpGammaiI})$ so that $\ran(f) \subset A_j$. 
\end{lemma}
\begin{proof}[Proof Sketch.]

This lemma can be proved in a roughly identical way to the way
{\cite[Theorem 1.2]{M2}} generalizes {\cite[Lemma 2.1]{M1}}. 
Similarly to that proof, either player I has a winning strategy in the game
above
associated to the complement of $A_0$, viewing $\fpGammaiI$
as a free product of the two groups $\Gamma = \fp_{i \neq 0}
\Gamma_i$ and $\Delta = \Gamma_0$, or else there is some $j > 0$ so that player II has a winning
strategy in the game associated to $A_j$, viewing
$\fpGammaiI$ as a free product of the two groups $\Gamma = \Gamma_j$
and $\Delta = \fp_{i \neq j} \Gamma_i$. (This is because if not, playing winning
strategies for the other players in all these games simultaneously would
yield some $y \in H((\fpGammaiI)^{\fpGammaiI})$ not in any $A_i$,
contradicting the fact that $\{A_i\}_{i \in I}$ partitions this set). One
then copies the construction from the proof of Lemma~\ref{hyp_lemma} above. 
\end{proof}

In a different direction, we could work instead with a universal free
hyperfinite action of $\Gamma$ (in the sense of {\cite[Section 2.5]{JKL}} and
\cite{CK}) instead of the action we have used on $H(\Gamma^\Gamma)$. Using
this universal action, Lemma~\ref{hyp_lemma} would remain true using a very
similar proof. 

There is a different way of viewing the action of $\Gamma$ on
$H(\Gamma^\Gamma)$:

\begin{remark}
Suppose $\Gamma$ and $\Delta$ are countable discrete groups. Then $\Gamma$
acts on $\Delta^\Gamma$ via 
\[\gamma \cdot x(\gamma') = x(\gamma^{-1})^{-1} x(\gamma^{-1} \gamma').\]
Let $H'(\Delta^\Gamma)$ be the set of $x \in \Delta^\Gamma$ such that
$x(1_\Gamma) = 1_\Delta$, $x$ is a bijection, and $x \in
\Free(\Delta^\Gamma)$ and $x^{-1} \in \Free(\Gamma^\Delta)$. 
Let $E_{\Gamma,\Delta}$ be the orbit equivalence relation of
the action of $\Gamma$ on $H'(\Delta^\Gamma)$. Then it is easy to see that $E_{\Gamma,\Delta}$ and
$E_{\Delta,\Gamma}$ are Borel isomorphic via the map sending $x \in
H'(\Delta^\Gamma)$ to $x^{-1} \in H'(\Gamma^\Delta)$. Hence, $E_{\Gamma,\Delta}$ is
generated by free actions of both $\Gamma$ and $\Delta$.
If $\Gamma$ is a
countably infinite group, the action of $\Gamma$ on $H(\Gamma^\Gamma)$ is
Borel isomorphic to the action of $\Gamma$ on $H'(\Z^\Gamma)$ via the
equivariant map sending $x \in H(\Gamma^\Gamma)$ to $f(x) \in
H'(\Z^\Gamma)$ where $f(x)(\gamma)$ is the unique $n \in \Z$ such that
$(x^{n}(1))^{-1} = \gamma^{-1}$. 
\end{remark}

\section{Corollaries}

\subsection{Colorings and matchings}

The proof of Theorem~\ref{comb_thm} is identical to the proofs in
\cite{M1}, simply replacing $\Free(\N^\Gamma)$ with $H(\Gamma^\Gamma)$ in
the definition of $G(\Gamma,\N)$ in that paper, and the proofs of Theorems
1.3, 1.4, and 1.5 of \cite{M1}. 

\subsection{The local lemma}
\label{lll_subsection}

Let us begin by recalling the Borel version of the Lov\'asz local lemma in \cite{CGMPT}. 
Suppose $G$ is a Borel graph on $X$, but where we allow loops so that 
we do not assume $G$ is irreflexive.
We use
the notation $G(x) = \{y \in X : x \mathrel{G} y\}$ to denote the
neighborhood of $x$. We also let $G^{\leq 2}$ be the Borel graph 
on $X$ where $x \mathrel{G^{\leq 2}} y$ if $d_G(x,y) \leq 2$.
(This graph is called $\mathrm{Rel}(G)$ in \cite{CGMPT}). 

Suppose $b \geq 1$. Then a \define{Borel $b$-local rule} $R$ for $G$ is a
Borel function whose domain is $X$ and where for each $x \in X$, $R(x)$ is
a set of functions from $G(x)$ to $b$. Say that $f \from X \to b$
\define{satisfies}
$R$ if $f \restriction G(x) \in R(x)$ for every $x \in X$. 
Define $p_R(x)$ to be the
probability that a random function from $G(x) \to b$ is not in $R(x)$. So
$p_R(x) = 1 - \frac{|R(x)|}{b^{|G(x)|}}$.

\begin{thm}[{\cite[Theorem 1.3]{CGMPT}}]
  Suppose $G$ is a Borel graph on $X$ so that $G^{\leq
  2}$ has uniformly subexponential growth and degree bounded by $\Delta$.
  If $R$ is a Borel $b$-local rule for $G$ such that $p_R(x) < \frac{1}{e
  \Delta}$ for all $x \in X$, then there exists a Borel function $f \from X
  \to b$ which satisfies $R$. 
\end{thm}

Theorem~\ref{lll_thm} clearly follows from the following:

\begin{lemma}
  Suppose $n \geq 6$, and 
  let $S$ be a free symmetric generating set for $\F_{2n}$, which acts on
  the space $H(\F_{2n}^{\F_{2n}})$ via Definition~\ref{action_defn}. Let 
  $G$ be the graph on $H(\F_{2n}^{\F_{2n}})$
  where $x \mathrel{G} y$ if there exists
  $\gamma \in S \union \{1\}$ such that $\gamma \cdot x = y$. 
  Then there
  exists a Borel $2$-local rule $R$ for $G$ such that $p_R(x) < \frac{1}{e
  \Delta}$ for all $x$, however there is no Borel function $f$ which
  satisfies $R$. 
\end{lemma}
\begin{proof}
  Partition the generating set $S$ into two symmetric sets $S_0$ and $S_1$ so that $S_0$ and
  $S_1$ generate two isomorphic copies $\Gamma_0$ and $\Gamma_1$ of $\F_n$,
  where $\Gamma_0 * \Gamma_1 = \F_{2n}$.
  
  Now let $R$ be the local rule where $f \in R(x)$ if 
  $f(x) = 0$ implies there is a $\gamma \in S_0$ such that $f(\gamma \cdot
  x) = 1$ and $f(x) = 1$ implies there is a $\gamma \in S_1$ such that
  $f(\gamma \cdot x) = 0$. 
  By Lemma~\ref{hyp_lemma}, for every Borel function 
  $f \from H(\F_{2n}^{\F_{2n}}) \to 2$, by viewing $f$ as the
  characteristic function of some set, there is either an entire
  $\Gamma_0$-orbit whose image is $\{0\}$ or a $\Gamma_1$-orbit whose image
  is $\{1\}$. Hence, there can be no Borel function $f$ satisfying $R$.
  
  However, for every $x$, we have $p_R(x) = 1/2^{2n}$ and the graph 
  $G^{\leq 2}$ has degree $1 + (4n)^2$. To finish, note that 
  \[\frac{1}{2^{2n}} < \frac{1}{e (1 + (4n)^2)}\]
  for $n \geq 6$. 
\end{proof}

\subsection{An exotic finitely additive invariant Borel measure}

\begin{proof}[Proof of Theorem~\ref{exotic_fin_measure}.]

  Consider the action of $\Gamma = (\Z/2\Z)^{* 3} =
  \<a,b,c:a^2=b^2=c^2=1\>$ on $X = H(\Gamma^\Gamma)$. $X$ is a Borel subset
  of the Polish space $\Gamma^\Gamma$, and so by changing topology, we may
  give a Polish topology to $X$ that has the same Borel sets but so that
  the action of $\Gamma$ on $X$ is continuous. Since $\Gamma$ is
  nonamenable and a free probability measure preserving action of a
  nonamenable group on a standard probability space $(X,\mu)$ cannot be
  $\mu$-hyperfinite, this action does not admit any countably additive
  invariant Borel probability measure. 
  
  Let $\mathcal{B}(X)$ be the $\sigma$-algebra of Borel subsets of $X$. Now
  $\mathcal{B}(X)$ is invariant under the action of $\Gamma$ and hence by
  {\cite[Theorem 9.1]{W}} there is a finitely additive $\Gamma$-invariant
  probability measure $\nu \from \mathcal{B}(X) \to [0,1]$ with $\mu(X) = 1$ if and
  only if $n+1$ copies of $X$ are not Borel equidecomposable with a subset
  of $n$ copies of $X$. So it suffices to show that for all $n \in \N$ and
  finite sets $S \subset \Gamma$ there do not exist $n+1$ Borel functions
  $f_0, \ldots, f_n$ such that for all $x \in X$ and $i \leq n$, $f_i(x) =
  \gamma \cdot x$ for some $\gamma \in S$ and for every $y \in X$, $\{(z,i)
  : f_i(z) = y\}$ has at most $n$ elements. 

  Suppose for a contradiction that there did exist such a finite set $S
  \subset \Gamma$ and Borel functions $f_0, \ldots, f_n$ as above. Let $G$ be the Borel graph on $X$
  where $x \mathrel{G} y$ if there is a generator $\gamma \in \{a^{\pm 1},
  b^{\pm 1}, c^{\pm 1}\}$ such that $\gamma \cdot x = y$. Note that $G$ is
  acyclic (since the action of $\Gamma$ is free, and the Cayley graph of
  $\Gamma$ with respect to its generators is acyclic), so there is a unique
  path between any two points in $G$ in the same connected component. Our
  rough idea is that the graph $G$ does not have a Borel antimatching (in
  the terminology of \cite{M1}) as one can easily see from
  Lemma~\ref{hyp_lemma}. However, one can construct a Borel antimatching
  assuming the existence of these functions $f_0, \ldots, f_n$. 

  Let $g$ be the
  Borel function which associates to each directed edge $(x,y)$ of $G$ the number of
  pairs of the form 
  $(z,i)$ where $z \in X$ and $i \leq n$ and the unique
  $G$-path from $z$ to $f_i(z)$ includes the directed edge $(x,y)$. Note that since $S$ is
  finite and $G$ has bounded degree, $g$ is bounded above. Now we
  claim that for every $x \in X$, there is some neighbor $y$ of $x$ such that
  $g((x,y)) > g((y,x))$. To see this, consider the quantity
  \[\sum_{\{y : y \mathrel{G} x\}} g((x,y)) - g((y,x)).\]
  Take a pair $(z,i)$ that contributes to this sum because the
  path from $z$ to $f_i(z)$ includes $x$. If $x \neq z$ and $x \neq f_i(z)$,
  then this path has one edge directed towards $x$ and one away from $x$,
  so the net contribution to the sum is zero.
  If $z = x$, then there are exactly $n+1$ pairs of the form $(x,i)$,
  and so $n+1$ edges directed away from $x$. However, if $f_i(z) = x$, then
  by assumption there are at most $n$ pairs of the form $(z,i)$ such that
  $f_i(z) = x$. Hence the total sum is positive, and so there must be some
  $y$ such that $g((x,y)) - g((y,x))$ is positive.
  
  Let $<$ be a Borel linear ordering of $X$. We now define a Borel function
  $h \from X \to X$ by setting 
  $h(x) = y$ where $y$ is the $<$-least neighbor of $x$
  such that $g((x,y)) - g((y,x)) > 0$. Note that $h^2(x) \neq x$ for every
  $x$. Now let $A_{\gamma} = \{x : h(x) = \gamma \cdot x\}$ for $\gamma
  \in \{a,b,c\}$ so these sets partition $H(\Gamma^\Gamma)$. Finally, by
  applying Lemma~\ref{hyp_lemma} twice (or Lemma~\ref{ctbl_version} once),
  there must be some $\gamma \in
  \{a,b,c\}$ so that if $\<\gamma\>$ is the subgroup generated by $\gamma$,
  there is a Borel injective $\<\gamma\>$-equivariant function $f \from
  H(\<\gamma\>^{\<\gamma\>}) \to A_\gamma \subset H(\Gamma^\Gamma)$. But
  any $y \in \ran(f)$ has $h(y) = \gamma \cdot
  y$ and $h(\gamma \cdot y) = y$, since both $y$ and $\gamma \cdot y$ are
  in $A_\gamma$. This contradicts the fact that $h^2(x) \neq x$ for all $x \in
  X$.
\end{proof}

\subsection{An ultrafilter on $\R/E_0$}

\begin{proof}[Proof of Theorem~\ref{uf}.]
  Instead of $E_0$, we will construct the ultrafilter $U$ on the equivalence
  relation $E_{\F_2}$ on $H(\F_2^{\F_2})$. Since $E_{\F_2}$ is hyperfinite,
  by \cite{DJK} it is Borel bireducible with $E_0$ restricted to some
  Borel subset of $2^\N$. Hence
  our construction will also yields an ultrafilter on the quotient of
  $E_0$.
  
  Fix $C$ as in the definition of the proof of Lemma~\ref{hyp_lemma} where
  $\Gamma = \Delta = \Z$ so $\Gamma * \Delta = \F_2$.
  Given an $E_{\F_2}$-invariant subset $A \subset H(\F_2^{\F_2})$, we
  define $A \in U$ if and only if player II wins the game $G_{A \union C}$ defined in
  the proof of Lemma~\ref{hyp_lemma}. 
  The proof that this defines an
  ultrafilter is identical to the proof of {\cite[Lemma 4.9]{M2}}.

  It is trivial to see that given a winning strategy for player II in the game
  $G$, then there are plays of the game using this winning strategy of
  every Turing degree above the Turing degree of this strategy. 
  Hence, given any subset of
  $A \subset H(\F_2^{\F_2})$ which is Turing invariant, $A$ is in the
  ultrafilter $U$ if and only if $A$ contains a Turing cone. Thus, $U$ is
  Rudin-Kiesler above Martin's measure, as witnessed by the identity
  function (which is a homomorphism from $E_{\F_2} \restriction
  H({\F_2}^{\F_2})$ to $\equiv_T$ on ${\F_2}^{\F_2}$, which we can identify
  with $\baire$). 
\end{proof}

\section{Lower bounds on component size in witnesses to hyperfiniteness}

We begin with a lemma about forward recurrent sets for bounded-to-one Borel
functions. 

\begin{lemma}\label{bfr}
  Suppose $f \from X \to X$ is a bounded-to-one Borel function. Then
  there is a Borel set $A \subset X$ such that $x \in A \implies f(x)
  \notin A$, and for all $x \in X$, one of $x, f(x), f^2(x), f^3(x)$ is in
  $A$. 
\end{lemma}
\begin{proof}
  By {\cite[Corollary 5.3]{KST}}, there is a Borel $3$-coloring $c \from X
  \to 3$ of $G_f$. Let $A$ be the set of $x \in X$ such that either 
  $c(x) = 0 \land c(f(x)) = 2$ or $c(x) = 1 \land
  c(f(x)) = 2$ or $c(x) = 0 \land c(f(x)) = 1 \land c(f^2(x)) = 0$. 
  It is easy to see that $A$ is $G_f$-independent. Given $x \in X$, the
  sequence $c(x), c(f(x)), c(f^2(x)), \ldots$ begins with either 
  \begin{itemize}
  \item 01, which continues 010 so $x \in A$ or 012 so $f(x) \in A$.
  \item 02, so $x \in A$.
  \item 10, which continues 102 so $f(x) \in A$ or 1012 so $f^2(x) \in A$ or 1010 so
  $f(x) \in A$. 
  \item 12, so $x \in A$. 
  \item 20, which continues 202 so $f(x) \in A$ or 2010 so $f(x) \in A$ or 2012
  $f^2(x) \in A$.
  \item 21, which continues 212 so $f(x) \in A$ or 2102 so $f^2(x) \in A$
  or 21010 so $f^2(x) \in A$ or 21012 so $f^3(x) \in A$.
  \end{itemize}
\end{proof}

We note that in general to answer Question~\ref{q1} positively for a Borel
graph $G$, it suffices to obtain a positive answer for an induced subgraph of $G$
which has exactly one connected component in each connected component of the original graph. 

\begin{lemma}\label{subgraph}
  Suppose $G$ is a bounded degree Borel graph on $X$, and $A \subset X$ is a Borel set
  such that every connected component of $G$ contains exactly one
  connected component of $G \restriction A$. Then if Question~\ref{q1} has
  a positive answer for $G \restriction A$, then Question~\ref{q1} has a
  positive answer for $G$. 
\end{lemma}
\begin{proof}

Fix a Borel linear ordering of $X$. We will begin by defining a graph $H$
with the same connectedness relation as $G$.
Let $H' \subset G$ be the graph on $X$
where $x \mathrel{H'} y$ if $x \mathrel{G} y$ and 
the edge $\{x,y\}$ is contained in the lex-least path
from either $x$ to $A$ or $y$ to $A$. Using properties of the lex-least ordering, it is easy
to see that $H'$ is acyclic, and each connected component of $H'$ contains
exactly one element of $A$. Let $H$ be the union of $H'$ and $G
\restriction A$.

Let $H'' \subset H'$ be the graph where $x
\mathrel{H''} y$ if $x \mathrel{H'} y$ and there are only finitely many $z$
so that the lex-least path from $z$ to $A$ includes the edge $\{x,y\}$. By
K\"onig's lemma, all the connected components of $H''$ are finite.
Let $m(\{x,y\}) = \max(d(x,A),d(y,A))$.

Now let $G_0' \subset
G_1' \subset \ldots$ be the hypothesized witness to the hyperfiniteness of $G
\restriction A$. 
We can define a witness $H_0 \subset H_1 \subset \ldots$ to the
hyperfiniteness of $H$ by setting $x
\mathrel{H_i} y$ if (i) $x \mathrel{G_i'} y$, or (ii) 
$x \mathrel{H'} y$ and $2^{i}
\nmid m(\{x,y\})$, or (iii) $x \mathrel{H''} y$. Clearly $H_0 \subset H_1
\subset \ldots$. 
We will check that each connected component $C$ of $H_i$ is finite
and contains at least $i$ elements.

First, suppose $C$
contains no element of $A$. Then $C$ contains a unique element $x_0$ that is
closest to $A$ since $H'$ is acyclic. Let $x_1$ be the unique neighbor of
$x_0$ such that $d(x_1, A) < d(x_0, A)$. Then $x_0$ is not $H_i$-adjacent to $x_1$ by
definition, and so $x_0$ and $x_1$ are not $H''$-adjacent by (iii).  By (ii),
we therefore must have that $d(x_0, A) = k 2^i$ for some $k$. 
By (iii) there must be infinitely many $z$ such that the
lex-least path from $z$ to $A$ includes the edge $\{x_0, x_1\}$. So by
K\"onig's lemma, there is some $H'$ path of length $2^i - 1$ from $x_0$ to
some point $z$ where $d(z,A) > d(x_0,A)$ so that this path does not use any
$H''$-edges. Thus, this path lies inside $H_i$, which therefore has at
least $2^i$ many elements. This suffices since $2^i \geq i$. Now if 
$x,y \in C$ and $x \mathrel{H_i} y$ but $x$ and $y$ are not $H''$-adjacent, then 
we see that $d(x,x_0) < 2^i$ and $d(y,x_0) < 2^i$ by (ii). Thus, there are
finitely many edges in $H_i \restriction C$ coming from condition (ii), and
so $H_i \restriction C$ is the
union of these edges with the finitely many $H''$-components that are
incident to them by condition (iii). So $C$ is finite.

Second, suppose $C$ does contain an element of $A$. Then $C \inters A$ is a
connected component of $G_i'$ since each $H'$-component contains only one
element of $A$. Thus, since each $G_i'$ component has at least $i$ many
elements, $C$ also has at least $i$ many elements. Now if $x, y \in C$ and
$x \mathrel{H_i} y$ but $x$ and $y$ are not $H''$-adjacent, then $m(\{x,y\})
< 2^i$ by (ii). Hence, $H_i \restriction C$ contains finitely many edges
coming from condition (ii) and also from (i) by above, and so $H_i
\restriction C$ is the union of these edges with the finitely many
$H''$-components that are incident to them by condition (iii). So $C$ is
finite.

Finally, we can define the desired witness $G_0 \subset G_1 \subset \ldots$ to the
hyperfiniteness of $G$ by
setting $x \mathrel{G_i} y$ if $x \mathrel{G} y$ and $x$ and $y$ are in
the same connected component of $H_i$.
\end{proof}

Now given a bounded degree Borel graph $G$ on a standard Borel space $X$,
if $G' \subset G$ is a subgraph of $G$ with finite connected components,
then we can form the \define{graph minor} $G/G'$ of $G$ by the
connectedness relation of $G'$. That is, the vertex set of this minor is
the standard Borel space $X/G'$ of connected components of $G'$, and the
edge relation of $G/G'$ is defined by $[x]_{G'} \mathrel{G/G'} [y]_{G'}$ if
$[x]_{G'} \neq [y]_{G'}$ and there exists $x' \in [x]$ and $y' \in
[y]_{G'}$ such that $x \mathrel{G} y$. Let $H = G/G'$ and suppose now that
$H' \subset H$ is a subgraph of $H$ with finite connected components. Then
$H'$ naturally lifts to a subgraph of $G$ with finite connected components that
contains $G'$. That is, there is an edge in this lifted graph between $x$
and $y$ if $x \mathrel{G'} y$, or $x \mathrel{G} y$ and $[x]_G'
\mathrel{H'} [y]_G'$. In several of our proofs below, we will define
iterated sequences of graph minors in this way, which will naturally lift
to witnesses of the hyperfiniteness of the original graph.

\begin{lemma}\label{bounded_to_one}
  Suppose $f \from X \to X$ is bounded-to-one Borel function that
  induces the graph $G_f$. Then Question~\ref{q1} has a
  positive answer for $G_f$.
\end{lemma}
\begin{proof}
  Let $f_0 = f$ and $X_0 = X$. Given the bounded-to-one function $f_i$ on
  $X_i$, let $A_i \subset X_i$ be as in Proposition~\ref{bfr}, and let
  $G'_{i} \subset G_{f_i}$ be the graph on $X_i$ with finite connected components
  where $x \mathrel{G'_i} f(x)$ if $x \notin A_i$. Note that 
  every
  connected component of $G_f$ has size at least $2$, and size at most $1 +
  d + d^2 + d^3$, if $f_i$ is $\leq d$-to-one. 
  
  Let $X_{i+1} = X_i/{G'_{i}}$ and for each $x \in X$, let $[x]_{i+1} \in
  X_{i+1}$ be the representative of $x$ in $X_{i+1}$, so $[x]_{i+1}$ is a
  finite set of elements of $X_i$, one of which is $[x]_i$. Let
  $f_{i+1}$ be the bounded-to-one Borel function on $X_{i+1}$ where
  $f_{i+1}([x]_{i+1}) = [y]_{i+1}$ if $[x]_{i+1} \neq [y]_{i+1}$ and there
  are $[x']_i \in [x]_{i+1}$ and $[y']_i \in [y]_{i+1}$ such that
  $f_i([x']_i) = [y']_i$. Note that $G_{f_{i+1}}$ is equal to the graph
  minor $G_{f_i}/{G'_i}$. The sequence $G'_0, G'_1, \ldots$ lifts to an
  increasing union $G_0'' \subset G_1'' \subset \ldots$ of Borel graphs on
  $X$. By induction, the connected components of each $G_i''$ are finite,
  and have size at least $2^i$.

  Let $H = G_f \setminus \bigunion_i G_i''$, so that $x \mathrel{H}
  f(x)$ if and only if $[x]_i \in A_i$ for every $i$. Let $G_i$ be
  the graph on $X$ where $x \mathrel{G_i} y$ if $x \mathrel{G_i''} y$ or $x
  \mathrel{H} y$. Then clearly $G_0 \subset G_1 \ldots$, every connected
  component of $G_i$ has at least $2^i \geq i$ elements (since this is true
  of $G_i''$), and $\bigunion_i G_i = G_f$. 
  We just need to show that every connected component of $G_i$ is finite.

  Let $H_{i}$ be the graph on $X_{i}$ where $[x]_{i} \mathrel{H_{i}}
  [y]_{i}$ if $[x]_i \in A_{i}$ and $f_{i}([x]_i) = [y]_i$ or $[y]_i \in
  A_{i}$ and $f_{i}([y]_{i}) = [x]_i$. Since $f_i$ is finite-to-one,
  by the definition of $A_i$, it is easy to see that every $H_i$ class is
  finite. Now if $x \in X$, then the 
  $G_i''$-class of $x$ is $\{y
  \in X \colon [y]_{i+1} = [x]_{i+1}\}$ by the definition of $X_{i+1}$.
  Thus, the $G_i$-class of $x$ is a subset
  of $\{y \in X : [y]_{i+1} \text{ is in the same $H_{i+1}$-class as
  }[x]_{i+1}\}$, which is clearly finite since $H_{i+1}$ has finite
  connected components.
\end{proof}

\begin{lemma}\label{two-ends}
  Suppose $G$ is a Borel graph on $X$ where every connected component of $G$ has
  two ends. Then Question~\ref{q1} has a positive answer for $G$.
\end{lemma}
\begin{proof}
Let $Y \subset [X]^{< \infty}$ be the collection of finite connected sets
$C \subset X$ such that removing $C$ from $G$ disconnects the connected
component containing $C$ into exactly two infinite pieces. Using a
countable Borel coloring of the intersection graph on $Y$ (see {\cite[Lemma
7.3]{KMi}} and {\cite[Proposition 2]{CM}}), we may find a Borel set $Z \subset
Y$ of pairwise disjoint subsets of $X$ which meets every connected
component of $G$. Let $G'$ be the graph on $Z$ where $C_0 \mathrel{G'} C_1$ if
$C_0$ and $C_1$ are in the same connected component of $G$ and there is no $D
\in Z$ such that removing $D$ from $G$ places $C_0$ and $C_1$ in different
connected components. The graph $G'$ has degree at most $2$, and it clearly
suffices to show that Question~\ref{q1} has a positive answer for $G'$
instead of $G$. 

Thus, we may restrict our attention just to $2$-regular
acyclic Borel graphs. 
However, this is trivial by using the existence of maximal Borel
independent sets {\cite[Proposition 4.6]{KST}} for such graphs, and the same idea as
Lemma~\ref{bounded_to_one}. 
\end{proof}

We are now ready to prove Proposition~\ref{q1_prop}.

\begin{proof}[Proof of Proposition~\ref{q1_prop}.]
Suppose $G$ is a hyperfinite bounded degree Borel graph on a standard Borel
space $X$, and $\mu$ is a Borel probability measure on $X$. Then
Question~\ref{q1} has a positive answer modulo a $\mu$-nullset, by using
Adams's end selection theorem {\cite[Lemma 3.21]{JKL}}, and the
Lemmas~\ref{subgraph}, \ref{bounded_to_one}, and \ref{two-ends}. 

Question~\ref{q1} has a positive answer modulo meager sets
via an easy Kuratowski-Ulam argument. 
\end{proof}

\end{document}